\documentclass[a4paper,12pt,reqno,oneside]{amsart}
\usepackage[utf8]{inputenc}
\usepackage{amsmath,amssymb}
\usepackage[numbers]{natbib}
\usepackage{dsfont}
\usepackage{array}
\usepackage{accents}
\usepackage{booktabs}
\usepackage{indentfirst}
\usepackage{eucal}
\usepackage[a4paper,top=3cm,bottom=3cm,left=3cm,right=3cm,headsep=18pt]{geometry}
\usepackage{graphicx}
\usepackage{randtext}
\usepackage{url}
\usepackage{hyperref}

\theoremstyle{plain}
\newtheorem{teo}{Theorem}[section]
\newtheorem{cor}{Corollary}[section]
\newtheorem{lem}{Lemma}[section]
\newtheorem{prop}{Proposition}[section]

\theoremstyle{definition}
\newtheorem{defn}{Definition}[section]

\theoremstyle{remark}
\newtheorem*{obs}{Remark}

\numberwithin{equation}{section}
\newcolumntype{P}[2]{>{#1\arraybackslash}m{#2}}

\newcommand{\bbZ}{\mathbb{Z}}
\newcommand{\bbN}{\mathbb{N}}
\newcommand{\bbT}{\mathbb{T}}

\newcommand{\GG}{\mathbb{G}}
\newcommand{\VV}{\mathbb{V}}
\newcommand{\EE}{\mathbb{E}}
\newcommand{\GGO}{\overset{\rightarrow}{\mathbb{G}}}
\newcommand{\mcR}{\mathcal{R}}
\newcommand{\mcB}{\mathcal{B}}
\newcommand{\EEO}{\overset{\,\rightarrow}{\mathbb{E}}}
\newcommand{\vxy}{\overset{\,\rightarrow}{xy}}
\newcommand{\dsE}{\mathds{E}}
\newcommand{\dsP}{\mathds{P}}
\newcommand{\dd}{d_1,d_2}
\newcommand{\raiz}{\varnothing}
\newcommand{\dist}{\textnormal{dist}}

\newcommand{\YY}{\mathcal{Y}}
\newcommand{\ind}{\mathds{1}}
\newcommand{\rlinf}{\underaccent{\bar}{r}}
\newcommand{\rlsup}{\bar{r}}

\newcommand{\FM}[2]{\ensuremath{\textnormal{FM}(#1,#2)}}
\newcommand{\APM}[2]{\ensuremath{\textnormal{APM}(#1,#2)}}

\newcommand{\pc}[1]{\ensuremath{{p_c}(#1)}}

\newcommand{\s}[2]{\ensuremath{#1 \rightarrow #2}}
\newcommand{\ns}[2]{\ensuremath{#1 \nrightarrow #2}}

\newcommand{\p}[3]{\ensuremath{#1 \stackrel {#3}{\rightarrow} #2}}
\newcommand{\np}[3]{\ensuremath{#1 \stackrel {#3}{\nrightarrow} #2}}
\newcommand{\m}[3]{\ensuremath{#1 \stackrel {#3}{\leadsto} #2 }}

\begin{document}

\title[Frogs on biregular trees]{Phase transition for the frog model\\
on biregular trees}

\author{Elcio Lebensztayn}
\address[E. Lebensztayn]{Institute of Mathematics, Statistics and Scientific Computation\\
University of Campinas -- UNICAMP\\
Rua S\'ergio Buarque de Holanda 651, 13083-859, Campinas, SP, Brazil.}
\email[E. Lebensztayn]{\randomize{lebensz@unicamp.br}}
\thanks{The authors are thankful to the National Council for Scientific and Technological Development -- CNPq (Jaime Utria's grant No.~140887/2017-2), and the S\~ao Paulo Research Foundation -- FAPESP (Grant No.~2017/10555-0).}

\author{Jaime Utria}
\address[J. Utria]{Institute of Mathematics and Statistics\\
Fluminense Federal University -- UFF\\
Rua Professor Marcos Waldemar de Freitas Reis s/n, 24210-201, Niter\'oi, RJ, Brazil.}
\email[J. Utria]{\randomize{jutria@id.uff.br}}

\date{}



\begin{abstract}
We study the \emph{frog model with death} on the biregular tree $\mathbb{T}_{d_1,d_2}$. 
Initially, there is a random number of active and inactive particles located on the vertices of the tree. 
Each active particle moves as a discrete-time independent simple random walk on $\mathbb{T}_{d_1,d_2}$ and has a probability of death $(1-p)$ before each step. 
When an active particle visits a vertex which has not been visited previously, the inactive particles placed there are activated.
We prove that this model undergoes a phase transition: for values of $p$ below a critical probability $p_c$, the system dies out almost surely, and for $p > p_c$, the system survives with positive probability. 
We establish explicit bounds for $p_c$ in the case of random initial configuration.
For the model starting with one particle per vertex, the critical probability satisfies $p_c(\mathbb{T}_{d_1,d_2}) = 1/2 + \Theta(1/d_1+1/d_2)$ as $d_1, d_2 \to \infty$.
\end{abstract}

\maketitle

\baselineskip=20pt

\section{Introduction}

This paper addresses the issue of phase transition for the \textit{frog model with death}, a discrete-time growing system of simple random walks on a rooted graph $\GG$, which is described as follows.
Initially there is an independent random number of particles at each vertex of~$\GG$. 
All particles are inactive at time zero, except for those that might be placed at $\raiz$, the root of $\GG$. 
Each active particle moves as a discrete-time independent simple random walk (SRW) on the vertices of $\GG$, and has a probability of death $(1-p)$ before each step. When an active particle visits an inactive particle, that particle becomes active and starts to walk, performing exactly the same dynamics, independently of everything else. 
In the literature, the particles are often referred to as frogs, whose possible states are described as sleeping (inactive) and awake (active). 
This process can be thought as a model for describing rumor (or infection) spreading. 
One can think of every awake frog as an informed (or infected) agent, which shares the rumor with (or infects) a sleeping frog at the first time they meet.

In the last decades, there has been a growing interest in understanding the behavior of stochastic processes on more general graph structures than the $d$-dimensional integer lattice $\bbZ^d$ and the homogeneous tree~$\bbT_d$ of degree $(d + 1)$.
In particular, the study of phase transitions and critical phenomena for stochastic systems on nonhomogeneous trees and nonamenable quasi-transitive graphs has become an important research area.
Among the models studied, are percolation, the contact process, and branching random walks. 
For the frog model, most of the work has involved studying the process on $\bbZ^d$, $\bbT_d$ and recently on $d$-ary trees.
As far as we know, only \citet{JR} considers another kind of tree, proving the recurrence of the process (without death) on a $3,2$-alternating tree (in which the generations of vertices alternate between having $2$ and $3$ children).

The first published paper dealing with the frog model (with $p=1$, $\GG=\bbZ^d$) is due to \citet{TW}, where it was referred to as the \textquotedblleft egg model\textquotedblright. They proved that, starting from the one-particle-per-vertex initial configuration, almost surely infinitely many frogs will visit the origin for all $d \geq 3$ (that is, although each frog is individually transient, the process is recurrent).
\citet{FIRE} exhibits the critical rate at which the frog model with Bernoulli($\alpha/||x||^2$) sleeping frogs at each $x\in\bbZ^d\setminus\{0\}$ changes from transience to recurrence. 
A similar result is obtained by \citet{JJHa} for the model on $d$-ary trees, with Poisson($\mu$) sleeping frogs at each vertex. 
More precisely, the authors prove that the model undergoes a phase transition between transience and recurrence, as the initial density $\mu$ of frogs increases. 
For the model starting with one frog per vertex, \citet{JJHb} establish that there is a phase transition in the dimension of the tree, by proving recurrence for $d=2$ and transience for $d\geq 5$.
Based on simulations, they conjecture that the model is recurrent for $d=3$, and transient for $d=4$. 

In \citet{ST}, for the frog model without death on $\bbZ^d$, it is proved that, starting from the one-particle-per-vertex initial configuration, the set of the original positions of all awake frogs, rescaled by the elapsed time, converges to a nonempty compact convex set.
\citet{STR} prove the same statement in the case of random initial configuration; these results are known as shape theorems. 
For a continuous-time version of the frog model, a limiting shape result is stated by \citet{STCT}. 
We refer to \citet{FIS} for a survey on some results for the model and its variations.

Regarding the frog model with death, the existence of phase transition as $p$ varies was first studied by \citet{PT}.
As we will detail later, the occurrence of phase transition means that there is a nontrivial value $p_c$ of $p$ separating the regimes of extinction and survival of the process.
\citet{PT} obtain lower bounds on~$p_c$ for general graphs and bounded degree graphs, and establish the existence of phase transition on $\bbZ^d$ and $\bbT_{d}$, for $d\geq 2$, under rather broad conditions.
\citet{IUB} prove that the critical probability for the frog model on a homogeneous tree of degree $(d + 1)$ is at most $(d + 1)/(2d)$; that result is an improvement of the upper bound stated by \citet{Mono}, namely, $(d + 1)/(2d - 2)$.
Further improvements on the upper bound for this critical probability were recently obtained by \citet{FMRT}, using Renewal Theory, and by \citet{NUB}.
For more details on the subject, see these papers and references therein.

The aim of the present paper is to deepen the study of the critical phenomenon of the frog model, going beyond $\bbZ^d$ and homogeneous trees.
We consider the model on a specific class of nonhomogeneous trees, namely, biregular trees. 
In the main results, we prove a sufficient condition for the existence of phase transition, and present explicit bounds for the critical probability, in the case of random initial configuration.
Since there is more than one parameter measuring the size of such trees, bounds on the critical parameter are harder to get at.
The asymptotic behavior of the critical probability for large values of the dimension of the tree is also derived.
We hope that our study will stimulate further works regarding the occurrence of phase transition on other nonhomogeneous trees and nonamenable graphs.

\section{Formal description of the model and main results }

We start off with some basic definitions and notation of Graph Theory. 
Let $\GG = (\VV, \EE)$ be an infinite connected locally finite graph, with vertex-set~$\VV$ and edge-set~$\EE$.
A vertex $\raiz \in \VV$ is fixed and called the \textit{root} of $\GG$.
We denote an unoriented edge with endpoints $x$ and $y$ by $xy$. 
Vertices $x$ and $y$ are said to be \textit{neighbors} if they belong to a common edge~$xy$; we denote this by $x\sim y$.
The \textit{degree} of a vertex is the number of its neighbors.
A \textit{path} of length~$n$ from $x$ to $y$ is a sequence $x=x_0, \dots, x_n=y$ of vertices such that $x_i\sim x_{i+1}$ for all $i=0, \dots, n-1$. 
The \textit{graph distance} $\dist(x, y)$ between $x$ and $y$ is the minimal length of a path connecting the two vertices; the \textit{level} of $x$ is $\dist(\raiz, x)$. 
A \textit{tree} is a connected graph without loops or cycles, where by a cycle in a graph we mean a sequence of vertices $x_0, \dots, x_n$, $n\geq 3$, with no repetitions besides $x_n=x_0$. 
A graph is \textit{bipartite} if its vertex-set~$\VV$ can be partitioned into two subsets $\VV_1$ and $\VV_2$, in such a form that every edge joins a vertex of $\VV_1$ to a vertex of $\VV_2$.
For $d_1 \geq 1$ and $d_2 \geq 1$, we denote by $\bbT_{\dd}$ the $(\dd)$-\textit{biregular tree}, which is the bipartite tree where the degree of a vertex is $(d_1+1)$ or $(d_2+1)$, according to whether the level of the vertex is even or odd.
In this case, the class $\VV_1$ (\textit{resp.}~$\VV_2$) is the set of vertices at even (\textit{resp.}~odd) distance from the root.
From now on, a vertex $x\in\VV_{i}$ will be called a \textit{type~$i$ vertex}.
Notice that $\bbT_{1, 1}$ is isomorphic to $\bbZ$. 
See Figure \ref{Fg: bitree} for an illustration of $\bbT_{2, 4}$.

\vspace{0.4cm}
\begin{figure}[ht]
\centering
\includegraphics[scale=0.6]{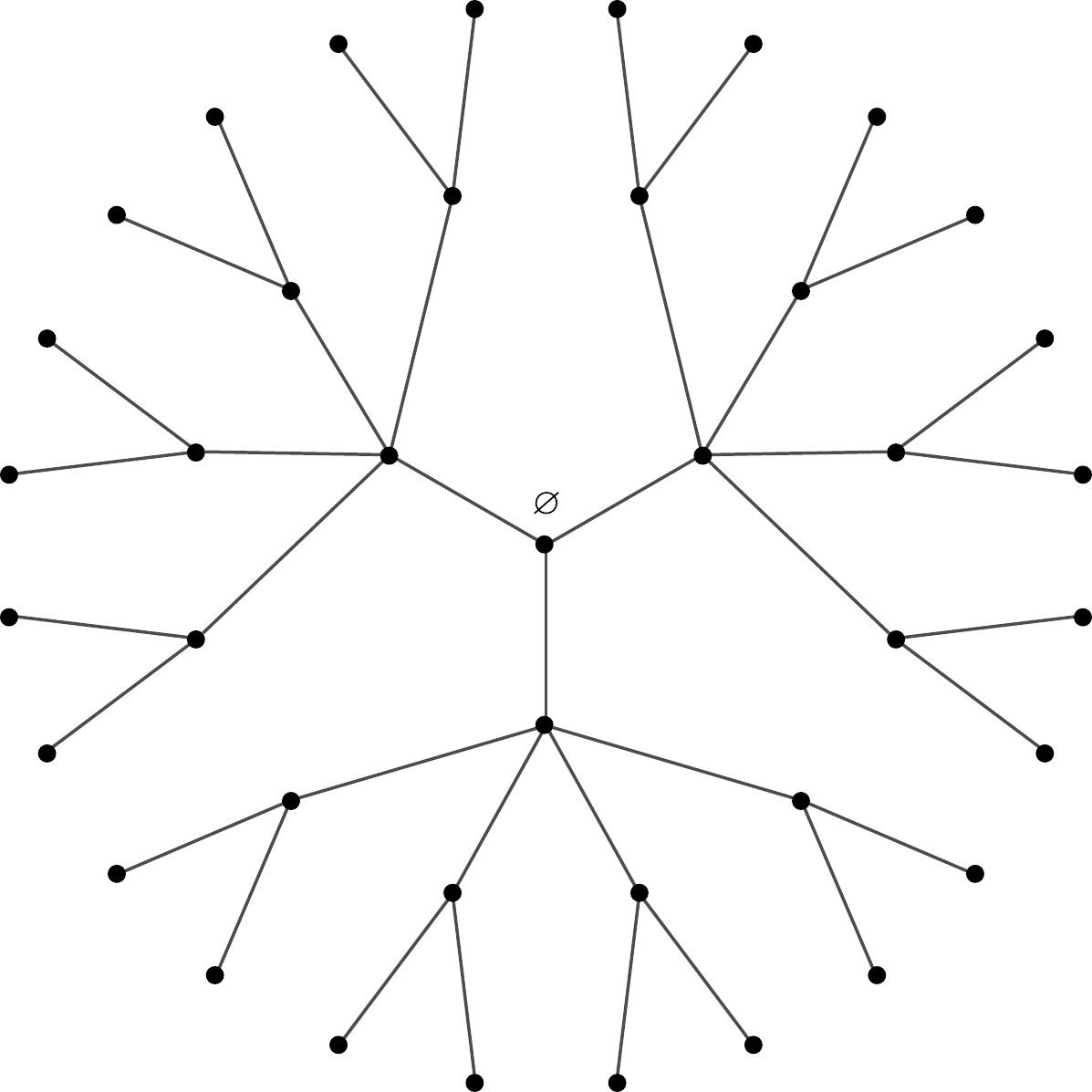}
\caption{The first three levels of the biregular tree $\bbT_{2, 4}$.}
\label{Fg: bitree}
\end{figure}

Now we describe the frog model in a formal way, keeping the notation of \citet{PT} and \citet{IUB}, whenever possible. 
We write $\bbN=\{1, 2, \dots\}$ and $\bbN_0=\bbN\cup \{0\}$. 
Let $\eta$ be a random variable assuming values in $\bbN_0$, and define $\rho_{k}:=\dsP[\eta=k]$, $k \in\bbN_0$.
We suppose that $\dsP[\eta \geq 1] > 0$, that is, $\rho_0<1$. 
For $s \in [0, 1]$, let $\varphi(s) := \dsE[s^\eta]$ be the \textit{probability generating function} of $\eta$.
To define the frog model, let $\{\eta(x):x\in\VV\}$, $\{(S_n^x(k))_{ n\in\bbN_0}; \; k\in\bbN, \, x\in\VV\}$ and $\{(\Xi_p^x(k));\;k\in\bbN, \, x\in\VV \}$ be independent sets of random objects defined as follows. For each $x\in\VV$, $\eta(x)$ has the same law as $\eta$, and gives the initial number of frogs at vertex~$x$. 
If $\eta(x)\geq 1$, then for each $k\in\{1, \dots, \eta(x)\}$, $(S_n^x(k))_{ n\in\bbN_0}$ is a discrete-time SRW on $\GG$ starting from $x$, and $\Xi_p^x(k)$ is a random variable whose law is given by $\dsP[\Xi_p^x(k)=j]=(1-p)p^{j-1}$, $j\in\bbN$, where $p \in [0, 1]$ is a fixed parameter. 
These random objects describe respectively the trajectory and the lifetime of the $k$-th frog placed initially at $x$.
Thus, the $k$-th frog at vertex $x$, whenever it is awoken, follows the SRW $(S_n^x(k))_{ n\in\bbN_0}$, and disappears at the instant it reaches a total of $(\Xi_p^x(k)-1)$ jumps. 
At the moment the frog disappears, it is not able to awake other frogs (first the frog decides whether or not to survive, and only after that it is allowed to jump). 
There is no interaction between awake frogs. 
At time $n = 0$, only the frogs that might be placed at the root of~$\GG$ are awake.
We call this model the \textit{frog model} on $\GG$, with survival parameter $p$ and initial configuration ruled by $\eta$, and denote it by $\FM{\GG}{p, \eta}.$ 

\begin{defn}
A particular realization of the frog model \textit{survives} if for every instant of time there is at least one awake frog. 
Otherwise, we say that it \textit{dies out}.
\end{defn}

A coupling argument shows that $\dsP[\FM{\GG}{p, \eta} \text{ survives}]$ is a nondecreasing function of~$p$, and therefore we define the \textit{critical probability} as 
\[\pc{\GG, \eta} = \inf\left\{p:\dsP[\FM{\GG}{p, \eta} \text{ survives}]>0\right\}. \]
As usual, we say that $\FM{\GG}{p, \eta}$ exhibits \textit{phase transition} if $ \pc{\GG, \eta} \in (0, 1)$.
For $\GG = \bbT_{\dd}$ and $\eta \equiv 1$ (one-particle-per-vertex initial configuration), we simply write $\pc{\bbT_{\dd}}$. 

As proved by \citet[Theorem~1.1]{PT}, if $d_1=d_2=1$, then under the condition $\dsE \log (\eta \vee 1)< \infty$, we have that the frog model dies out almost surely for every $p<1$, that is, $\pc{\bbZ, \eta}=1$. 
The picture is quite different for higher degrees. 
Indeed, \citet[Theorems~1.2 and 1.5]{PT} prove that the frog model on $\bbT_d$ exhibits phase transition for every $d\geq 2$, provided that $\rho_0<1$ and $\dsE\eta^\delta<\infty$ for some $\delta>0$. 
Here we extend this result to the case where the process lives on $\bbT_{\dd}$ (assuming $d_1\geq 2$ or $d_2 \geq 2$).

We begin by showing a sufficient condition to guarantee the almost sure extinction of the frog model on bounded degree trees for small enough $p$.

\begin{teo}
\label{T: SCE}
Let $\GG$ be an infinite tree of bounded degree, and suppose that there exists $\delta > 0$ such that $\dsE\eta^\delta<\infty$.
Then \(\pc{\GG, \eta} > 0\).
\end{teo}

\begin{obs}
This result is an extension of Theorem~1.2 in \citet{PT}, which is stated for a homogeneous tree~$\bbT_d$, $d \geq 2$.
Note, however, that one can not prove Theorem~\ref{T: SCE} by using a direct coupling and the corresponding result for homogeneous trees.
As shown by \citet{Mono}, in general, the critical parameter of the frog model is not a monotonic function of the graph.
A longstanding open question is whether, for a class of graphs (for example, trees), $\GG_1 \subset \GG_2$ in this class implies that $\pc{\GG_2,\eta} \leq \pc{\GG_1,\eta}$.
This problem was first raised by \citet{PT} (for $\bbZ^d$), and discussed again in \citet{FIS}, \citet{Mono} and \citet{IUB}.
\end{obs}

Under the condition that the mean number of frogs placed initially at every vertex is finite, we derive a nontrivial lower bound for the critical probability on $\bbT_{\dd}$.

\begin{teo}
\label{T: LB}
Suppose that $\dsE\eta<\infty$. 
If $d_1\geq 2$ or $d_2 \geq 2$, then 
\[\pc{\bbT_{\dd}, \eta}\geq \sqrt{\frac{(d_1+1)(d_2+1)}{[d_1(\dsE\eta+1) + 1][d_2(\dsE\eta+1) +1]}}.\]
\end{teo}

The proofs of Theorems~\ref{T: SCE} and \ref{T: LB} are given in Section~\ref{S: Extinction}.
To present an upper bound for $\pc{\bbT_{\dd}, \eta}$, we need the following definition.

\begin{defn} 
\label{D: tudo}
Let $q=1-\rho_0 > 0$, 
$\kappa=(d_1+1)(d_2+1)$, 
$\Delta =\kappa^2-2\kappa(d_1+d_2) p^2+(d_2-d_1)^2 p^4$. 
We also define the functions $\alpha, \beta, f:[0, 1] \rightarrow [0, 1]$ given by
{\allowdisplaybreaks
\begin{align}
\alpha(p)&=\alpha^{(\dd)}(p)=\left\{
\begin{array}{cl}
\dfrac{\kappa+p^2(d_2-d_1)-\sqrt{\Delta}}{2d_2(d_1+1)p} & \text{if } 0 < p \leq 1, \\[0.5cm]
0 & \text{if } p = 0, 
\end{array}	\right. \label{F: Prob-alpha}\\[0.2cm]
\beta(p)&=\beta^{(\dd)}(p)=\left\{
\begin{array}{cl}
\dfrac{\kappa+p^2(d_1-d_2)-\sqrt{\Delta}}{2d_1(d_2+1)p} & \text{if } 0 < p \leq 1, \\[0.5cm]
0 & \text{if } p = 0, 
\end{array}	\right. \label{F: Prob-beta}\\[0.25cm]
f^{(\dd, q)}(p)&=\alpha(p)\beta(p)[1+q(1-\alpha(p))][1+q(1-\beta(p))]-\frac{1}{d_1d_2}.
\label{F: Function-f}
\end{align}}%
\end{defn}

\begin{teo} 
\label{T: UBA}
Suppose that $\rho_0 < 1$. 
If $d_1\geq 2$ or $d_2 \geq 2$, then
\begin{equation*}
\pc{\bbT_{\dd}, \eta}\leq \tilde{p}(\dd, q), 
\end{equation*} 
where $\tilde{p}(\dd, q)$ is the unique root of the function $f^{(\dd, q)}$ in the interval $(0, 1)$.
\end{teo}

\begin{cor}
\label{C: UB bi-re}
Suppose that $\eta\equiv 1$. If $d_1 \geq 2$ or $d_2 \geq 2$, then 
\label{C: UBH}
\begin{equation*}
\pc{\bbT_{\dd}} \leq \frac{1}{2}\sqrt{\frac{(d_1 + 1)(d_2 + 1)}{d_1 d_2}}.
\end{equation*}
\end{cor}

The proofs of Theorem~\ref{T: UBA} and Corollary~\ref{C: UB bi-re} are presented in Section~\ref{S: Survival}.

\begin{obs}
In the case of $d_1=d_2=d$ and $\eta \equiv 1$ in Theorem~\ref{T: UBA} (also in Corollary~\ref{C: UB bi-re}), we obtain the upper bound for the critical probability $\pc{\bbT_d}$ that is proved in \citet[Section 4]{IUB}, namely, $(d+1)/(2d)$. 
\end{obs}

From Theorems~\ref{T: SCE} and \ref{T: UBA}, we conclude:

\begin{teo}
\label{T: PT}
Suppose that $\rho_0 < 1$ and $\dsE\eta^\delta<\infty$ for some $\delta > 0$. 
If $d_1\geq 2$ or $d_2 \geq 2$, then
\[ 0 < \pc{\bbT_{\dd}, \eta} < 1. \]
\end{teo}

For comparison, we include a result proved by \citet{PT}, that establishes a lower bound for the critical probability of the frog model on bounded degree graphs.
This result also appears in \citet[Lemma~1]{Mono}.

\begin{prop}[\citet{PT}]
\label{P: LB Alves}
Suppose that $\GG$ is a graph with maximum degree $(D + 1)$, and $\dsE\eta<\infty$.
Then, 
\begin{equation*}
\pc{\GG, \eta} \geq \dfrac{D+1}{D(\dsE\eta+1)+1}.
\end{equation*} 
\end{prop} 

Notice that, for $\GG=\bbT_{\dd}$ with $d_1=d_2$, the lower bound given in Theorem~\ref{T: LB} equals the one provided by Proposition~\ref{P: LB Alves}. 
However, for $d_1 \neq d_2$, the lower bound stated in Theorem~\ref{T: LB} is better than that of Proposition~\ref{P: LB Alves}.
For a numerical illustration, we consider $\eta \equiv 1$, and compute the bounds established in Theorems~{\ref{T: LB}, \ref{T: UBA}} and Proposition~\ref{P: LB Alves} for some values of the pair $(d_1, d_2)$. 
The resulting values are given in Table~\ref{NV}.

\begin{table}[ht]
\centering
\begin{tabular}{P{\raggedright}{1.8cm}P{\centering}{4cm}P{\centering}{4cm}P{\centering}{3.6cm}} 
\toprule
$(d_1, d_2)$ & LB Proposition \ref{P: LB Alves} &LB Theorem~\ref{T: LB} & UB Theorem~\ref{T: UBA} \tabularnewline
\midrule
(1, 2) & 0.6000 & 0.6325 & 0.8588 \tabularnewline%\hline
(1, 3) & 0.5714 & 0.6172 & 0.8039 \tabularnewline%\hline
(1, 4) & 0.5556 & 0.6086 & 0.7749 \tabularnewline%\hline
(2, 2) & 0.6000 & 0.6000 & 0.7500 \tabularnewline%\hline
(2, 3) & 0.5714 & 0.5855 & 0.7063 \tabularnewline%\hline
(2, 4) & 0.5556 & 0.5774 & 0.6828 \tabularnewline%\hline
(3, 100) & 0.5025 & 0.5359 & 0.5771 \tabularnewline%\hline
(3, 1000) & 0.5002 & 0.5347 & 0.5743 \tabularnewline%\hline
(4, 10000) & 0.5000 & 0.5271 & 0.5572 \tabularnewline
\bottomrule
\end{tabular}
\caption{Numerical values of the bounds on the critical probability $\pc{\bbT_{\dd}}$.}
\label{NV}
\end{table}

Now we present a result concerning the asymptotic behavior of $\pc{\bbT_{\dd}}$ as $d_1, d_2 \to \infty$. 
The lower and upper bounds given in Theorem~\ref{T: LB} with $\eta\equiv 1$ and Corollary~\ref{C: UB bi-re} are, respectively,
\[ \frac{1}{2} + \frac{1}{8} \left(\frac{1}{d_1}+\frac{1}{d_2}\right) + O\left(\frac{1}{d_1^2}+\frac{1}{d_2^2}\right)
\text{ and } \;
\frac{1}{2} + \frac{1}{4} \left(\frac{1}{d_1}+\frac{1}{d_2}\right) + O\left(\frac{1}{d_1^2}+\frac{1}{d_2^2}\right). \]
As a consequence, we find the correct order of magnitude for the critical probability.

\begin{cor} For the frog model on $\bbT_{\dd}$ starting from the one-particle-per-vertex initial configuration, we have 
\[\pc{\bbT_{\dd}}= \frac{1}{2} +\Theta\left(\frac{1}{d_1}+\frac{1}{d_2}\right) \quad 
\text{as } \, d_1, d_2 \to \infty.\]
\end{cor}

\section{Survival of the process}
\label{S: Survival}

In brief, the key idea to prove Theorem~\ref{T: UBA} is to describe $\FM{\bbT_{\dd}}{p, \eta}$ as a percolation model which dominates suitably defined Galton--Watson branching processes.
In Subsections~\ref{SS: FM Perc} and \ref{SS: BT}, we describe the frog model on a graph $\GG$ as a particular percolation model, and we compute its parameters when $\GG = \bbT_{\dd}$.
Then, in Subsection~\ref{SS: SBP}, we define a sequence of Galton--Watson branching processes which are dominated by the frog model on $\bbT_{\dd}$.
Finding $p$ such that each one of these branching processes is supercritical leads us to obtain a sequence of upper bounds for the critical probability, which converges to the upper bound~$\tilde{p}(\dd, q)$ stated in~Theorem~\ref{T: UBA}.
Subsection~\ref{SS: Proofs} is devoted to the finalization of the proof.

The technique to prove Theorem~\ref{T: UBA} is similar to that used by \citet{IUB} to establish an upper bound for the critical probability of the frog model on homogeneous trees.
However, here we have to deal with the computations related to the fact that the degrees of the vertices of the tree alternate.
Also to overcome this issue, we construct the sequence of approximating branching processes in a new manner, leading directly to the desired upper bound.
An analogous approach is used by \citet{NUB} to derive an improved upper bound for the critical parameter on homogeneous trees.

\subsection{\texorpdfstring{$\FM{\GG}{p, \eta}$}{FM(Td1d2, p, h)} seen as percolation}
\label{SS: FM Perc}

In the bond percolation model, each edge of an infinite locally finite graph $\bar \GG=(\bar \VV, \bar \EE)$ is randomly assigned the value $1$ (open) or $0$ (closed), according to some probability measure on the product space $\{0, 1\}^{\bar \EE}$. 
Then, one studies the connectivity properties of the random subgraph of $\bar\GG$ which arises by removing closed edges. 
For a fundamental reference on the subject, we refer to~\citet{Grim}.
 
Next, we describe the frog model on an infinite graph $\GG = (\VV, \EE)$ as a particular bond percolation model. 
Indeed, for every $x \in \VV$ and $1 \leq k \leq \eta(x)$, we define the virtual set of vertices visited by the $k$-th frog located originally at $x$ by
\[ \mcR_x^k:=\{S_n^x(k): 0\leq n<\Xi_p^x(k)\}\subset \VV. \]
The set $\mcR_x^k$ becomes real in the case when $x$ is actually visited (and thus the sleeping frogs placed there are awoken). 
We define the range of $x$ by 
\begin{equation*}
\mcR_x:=
\left\{
\begin{array}{cl}
\bigcup_{k=1}^{\eta(x)}\mcR_x^k &\text{if } \eta(x) \geq 1, \\[0.2cm]
\{x\} &\text{otherwise}. 
\end{array}	\right.
\end{equation*}

Now let $\GGO=(\VV, \EEO)$ be the oriented graph with vertex-set $\VV$ and edge-set $\EEO:=\{\vxy:(x, y) \in \VV \times \VV, x \neq y\}$. 
That is, for every pair of distinct vertices $x$ and $y$, an oriented edge is drawn from $x$ to $y$.
Then, we introduce the following notations: for $x, y\in \VV$ distinct, $[\s{x}{y}]:=[y\in \mcR_x]$, $[\ns{x}{y}]:=[y \notin \mcR_x]$. 

This defines an oriented dependent long range anisotropic percolation model on~$\GGO$: $\vxy$ is declared to be open if $[\s{x}{y}]$, and closed otherwise.
In general, the probability that an oriented edge is open depends on its orientation.
For instance, when $\GG = \bbT_{\dd}$, an edge emanating from a type~$1$ vertex to a type~$2$ vertex has probability of being open different from an edge emanating from a type~$2$ vertex to a type~$1$ vertex.
This feature gives the anisotropy of the oriented percolation model.
Notice also that the cluster of the root has infinite size if and only if there exists an infinite sequence of distinct vertices $\raiz=x_0, x_1, x_2, \dots$ such that $\s{x_{j}}{x_{j+1}}$ for all $j \geq 1$.
Of course, this event is equivalent to the survival of $\FM{\GG}{p, \eta}$. 

\subsection{Biregular trees}
\label{SS: BT}

Now let us consider $\GG = \bbT_{\dd}$.
Our next purpose is to obtain an explicit formula for the probability that an oriented edge is open. 
Toward this end, let $\alpha^{(\dd)}(p)$ (\textit{resp.}~$\beta^{(\dd)}(p)$) denote the probability that a type~$1$ (\textit{resp.}~type~$2$) frog ever visits a type~$2$ (\textit{resp.}~type~$1$) neighbor vertex. 
Avoiding a cumbersome notation, we sometimes write $\alpha(p)$ or simply $\alpha$ (\textit{resp.}~$\beta(p)$, $\beta$). 
With a typical abuse of language, we call the percolation model associated to $\FM{\bbT_{\dd}}{p, \eta}$ the \textit{anisotropic percolation model on the biregular tree} $\bbT_{\dd}$, with parameters $\alpha$, $\beta$ and ruled by $\eta$. 
We denote it by $\APM{\bbT_{\dd}}{\alpha, \beta, \eta}$.

The following lemma about hitting probabilities of SRWs on $\bbT_{\dd}$ provides a formula for $\alpha$ and $\beta$. 
Recall that $\VV_1$ and $\VV_2$ are respectively the set of vertices at even and odd distance from the root of~$\bbT_{\dd}$.

\begin{lem}
\label{L: Prob open edge}
Let $x\sim y$ be a pair of neighbor vertices, and suppose $(x, y)\in \VV_i \times\VV_j$, with 
$i, j \in \{1, 2\}$, $i \neq j$.
Then, 
\begin{equation*}
\dsP[y\in\mcR_x^1]=
\begin{cases}
\alpha(p) &\text{if } i=1, j=2, \\[0.2cm]
\beta(p)  &\text{if } i=2, j=1, 
\end{cases}
\end{equation*}
where $\alpha$ and $\beta$ are given in Definition \ref{D: tudo}.
\end{lem}

\begin{proof}%[Proof of Lemma~\ref{F: Prob open edge}]
Let $\tau_{ij}:=\tau_{xy}$ be the first time when the simple random walk on $\bbT_{\dd}$ starting from $x$ visits $y$. 
Suppose first that $p<1$, so conditioning on the lifetime of the frog located at $x$, we have 
\begin{equation*}
\dsP[y\in\mcR_x^1]=\dsE[p^{\tau_{ij}}].
\end{equation*} 
Now by conditioning on the first jump of the frog at $x$, we have 
\begin{align*}
\dsE[p^{\tau_{12}}]&=\frac{1}{d_1+1}p+\frac{d_1}{d_1+1}p\dsE[p^{\tau_{12}}]\dsE[p^{\tau_{21}}]\\
\dsE[p^{\tau_{21}}]&=\frac{1}{d_2+1}p+\frac{d_2}{d_2+1}p\dsE[p^{\tau_{12}}]\dsE[p^{\tau_{21}}]
\end{align*}

By right-continuity of the probability generating functions of $\tau_{12}$ and $\tau_{21}$, it follows that $\lim_{p\to 0^+} \dsE[p^{\tau_{12}}]=\lim_{p\to 0^+}\dsE[p^{\tau_{21}}]=0$. Therefore, \eqref{F: Prob-alpha} and \eqref{F: Prob-beta} are the only possible solutions for the previous system of equations.
This concludes the proof for $p<1$. 
Finally, if $p=1$, then
\[ \dsP[y\in\mcR_x^1]=\dsP[\tau_{ij}<\infty]=\lim_{p\to 1^-} \dsE[p^{\tau_{ij}}]=\frac{d_j+1}{d_j(d_i+1)}.\qedhere \]
\end{proof}

\begin{obs}
If $ d_1 \leq d_2 $, then $ \alpha^{(\dd)}(p) \geq \beta^{(\dd)}(p)$ .
\end{obs}

\begin{lem}
\label{L: Prob long open edge}
Let $x$ and $y$ be two vertices of $\bbT_{\dd}$ with $(x, y)\in \VV_i \times\VV_j$, $i, j \in \{1, 2\}$, and let $k = \dist(x, y) \geq 1$.
Let $\pi_\eta(i, j, k)$ denote the probability that the oriented edge $\vxy$ is open.
Then, 
\begin{equation}
\label{F: prob open edge aleator}
\pi_{\eta}(i, j, k)=
\begin{cases}
1-\varphi(1-\alpha^n\beta^{n-1}) 	&\text{if } i=1, j=2, k=2n-1, \\[0.1cm]
1-\varphi(1-\alpha^{n-1}\beta^n) 	&\text{if } i=2, j=1, k=2n-1, \\[0.1cm]
1-\varphi(1-\alpha^n\beta^n) 			&\text{if } i=j, k=2n, 
\end{cases}
\end{equation}
where $\varphi(s)=\dsE (s^\eta)$, $s\in [0, 1]$, is the probability generating function of $\eta$. 
\end{lem}

\begin{proof}%[Proof of Lemma~\ref{L: Prob long open edge}]
By conditioning on the initial number of frogs at vertex $x$, it is enough to consider $\eta \equiv 1$, that is, $\varphi(s) \equiv s$.
In this case, $\pi_{\eta}(i, j, k)$ is the probability that a type~$i$ frog ever visits a type~$j$ vertex at distance $k$. 
So, we need to prove that for $\eta \equiv 1$, 
\begin{equation}
\label{F: prob open edge}
\pi_{\eta}(i, j, k)=
\begin{cases}
\alpha^n\beta^{n-1} &\text{if } i=1, j=2, k=2n-1, \\[0.1cm]
\alpha^{n-1}\beta^n &\text{if } i=2, j=1, k=2n-1, \\[0.1cm]
\alpha^n\beta^n 		&\text{if } i=j, k=2n.
\end{cases}
\end{equation}

Formula \eqref{F: prob open edge} follows from the fact that if $(x, y)\in\VV_1\times\VV_2$ with $k=2n-1$, then the first time when the frog starting from $x$ visits $y$ is equal in distribution to the sum of $2n-1$ independent random variables, such that $n$ of them are independent copies of $\tau_{12}$ and $n-1$ of them are independent copies of $\tau_{21}$.
Using this fact and Lemma~\ref{L: Prob open edge}, we obtain that
\[ \pi_{\eta}(1, 2, 2n-1) = \alpha^n\beta^{n-1}. \]
The other cases are analogous.
%$\tau_{xy}$
%The second case (i.e., $(x, y)\in\VV_2\times\VV_1$) follows by symmetry. 
%In the case of $(x, y) \in \VV_i\times \VV_i$, $i=1, 2$, for $k=2n$, we have that $\tau_{xy}$ is equal to the sum of $2n$ independent random variables, such that, $n$ of them are independent copies of $\tau_{12}$ and $n$ are independent copies of $\tau_{21}$. 
\end{proof}

\subsection{A sequence of branching processes dominated by the frog model}
\label{SS: SBP}

The central idea is as follows. 
For each $n\geq 1$, we define a Galton--Watson branching process whose survival implies that the cluster of the root in $\APM{\bbT_{\dd}}{\alpha, \beta, \eta}$ has infinite size. 
For each branching process, we find a sufficient condition which guarantees that the process is supercritical. 
We get in this manner a sequence of upper bounds for the critical probability, which converges to the upper bound given in the statement of Theorem~\ref{T: UBA}. 
This technique of using embedded branching processes is very similar to that used for the contact process on homogeneous trees; see the paper by \citet{LSCP}.

Let us now carry out this plan.

\begin{defn}
\label{D: Sets}
\hfill
\begin{enumerate}

\item[\textnormal{(i)}] We define a partial order on the set $\VV$ as follows: for $x, y\in\VV$, we say that $x \preceq y$ if $x$ belongs to the path connecting $ \raiz $ and $y$; $x\prec y$ if $x \preceq y$ and $x \neq y$. 

\item[\textnormal{(ii)}] For any vertex $ x \neq \raiz $, let $ \VV^+(x) =\{y\in\VV: x\preceq y\}$.
Also define $ \VV^+(\raiz) = \VV \setminus \VV^+(z) $, where $z$ is a fixed vertex neighbor to~$ \raiz $.

\item[\textnormal{(iii)}] For $x \in \VV$ and $k \in \bbN$, define $ L_k (x) = \{ y \in \VV^+(x) : \dist (x, y) = k \} $.
\end{enumerate}
\end{defn}

\begin{defn}
\label{D: k Cam}
For $x\in \VV$ and $y\in L_{k}(x)$, $k\in \bbN$, consider $x_0=x \prec x_1 \prec \cdots \prec x_{k-1} \prec x_{k}=y$, the path connecting $x$ and $y$.
For each $\ell=1, 2, \dots, n-1$, we denote by $[\m{x_0}{x_{\ell}}{}]$ the event that $[\s{x_0}{x_{\ell}}] \cap [\ns{x_0}{x_{\ell+1}}]$.
We define the event $[\p{x_0}{x_{k}}{o}]$ inductively on~$k$ by:
\begin{itemize}
\item[\textnormal{(i)}] If $k=1$, then 
\[[\p{x_0}{x_k}{o}]:=[\s{x_0}{x_k}].\] 
\item[\textnormal{(ii)}] If $k \geq 2$, then
\[[\p{x_0}{x_{k}}{o}]:=[\s{x_0}{x_{k}}]\cup\bigcup_{\ell=1}^{k-1}[\m{x_0}{x_{\ell}}{}, \p{x_{\ell}}{x_{k}}{o}].\]
\end{itemize}
We denote the complement of $[\p{x}{y}{o}]$ by $[\np{x}{y}{o}]$.
\end{defn}

Now we construct a sequence of Galton--Watson branching processes embedded in $\APM{\bbT_{\dd}}{\alpha, \beta, \eta}$.
Starting from the root, the potential direct descendants of a vertex $x$ are vertices located in $L_k(x)$, where $k=2n$ is an even number.
Roughly speaking, a vertex $y\in L_{2n}(x)$ is said to be a child of vertex $x$ if and only if $[\p{x}{y}{o}]$. 
More formally, let $n\in\bbN$ be fixed.
Define $\YY_{0, n}:=\{\raiz\}$, and for $\ell\in\bbN$ define 
\[\YY_{\ell, n}:=\bigcup_{x\in \YY_{\ell-1, n}}\{y\in L_{2n}(x): \p{x}{y}{o}\}.\]
Let $Y_{\ell, n}=|\YY_{\ell, n}|$ be the cardinality of $\YY_{\ell, n}$. 

\begin{lem}
\label{P: BPdom}
For every $n\in\bbN$, $\{Y_{\ell, n}\}_{\ell\in\bbN_0}$ is a Galton--Watson branching process whose survival implies the occurrence of percolation. 
In addition, its mean number of offspring per individual satisfies 
$\dsE[Y_{1, n}]\geq (d_1 d_2)^n \, \phi^{(\dd, q)}_n(p)$, where
\begin{equation}
\label{E: prob offspring}
\phi^{(\dd, q)}_n(p) := q[\alpha(p)\beta(p)(1+q(1-\beta(p))]^n[1+q(1-\alpha(p))]^{n-1}.
\end{equation}
\end{lem}

Since the first claim in Lemma~\ref{P: BPdom} is clear, to prove it, we have to compute $\dsE[Y_{1, n}]$.
To accomplish this, we show a recursive formula for the probability of the event $[\p{x}{y}{o}]$.
Let us define the domain $\mathcal{A}_{\dd}:= \left[0, \frac{d_2+1}{d_2(d_1+1)}\right]\times\left[0, \frac{d_1+1}{d_1(d_2+1)}\right] \subseteq [0, 1]^2$. 
In formulas that appear from this point on, we assume that a summation of the form $\sum_{1}^0$ equals $0$.

\begin{lem}
\label{L: Prob cam open}
For $(x, y)\in \VV_i \times\VV_j$, $i, j \in \{1, 2\}$, with $\dist(x, y)=k \geq 1$, define
\[ \nu_{\eta}(i, j, k) := \dsP[\p{x}{y}{o}]. \]
%the probability of a type~$j$ vertex $y$ be offspring of a type~$i$ vertex $x$ at distance $k$.
For every $n \geq 1$, there exists functions $K_n$, $F_n$, $K_n^{\star}$, $F_n^{\star}$ with domain $\mathcal{A}_{\dd}$, not depending on $\dd$, such that
\begin{align*}
\nu_{\eta}(1, 2, 2n-1)&=K_n(\alpha(p), \beta(p)), 
&\nu_{\eta}(1, 1, 2n)&=F_n(\alpha(p), \beta(p)), \\[0.2cm]
\nu_{\eta}(2, 1, 2n-1)&=K_n^{\star}(\alpha(p), \beta(p)), 
&\nu_{\eta}(2, 2, 2n)&=F_n^{\star}(\alpha(p), \beta(p)).
\end{align*} 
\end{lem}

\begin{proof}%[Proof of Lemma~\ref{L: Prob cam open}]
For $n \geq 1$ and $(a, b)\in \mathcal{A}_{\dd}$, we define the following sequence of functions recursively:
{\allowdisplaybreaks
\begin{align*}
K_n(a, b) &= [1-\varphi(1-a^n b^{n-1})]
	\begin{aligned}[t]
	&+ \sum_{\ell=1}^{n-1} [\varphi(1-a^{\ell+1}b^\ell)-\varphi(1-a^\ell b^\ell)]K_{n-\ell}(a, b) \\
	&+ \sum_{\ell=1}^{n-1} [\varphi(1-a^\ell b^\ell)-\varphi(1-a^\ell b^{\ell-1})]F_{n-\ell}^{\star}(a, b), 
	\end{aligned}\\[0.2cm]
K^{\star}_n(a, b) &= [1-\varphi(1-b^n a^{n-1})]
	\begin{aligned}[t]
	&+ \sum_{\ell=1}^{n-1} [\varphi(1-b^{\ell+1}a^\ell)-\varphi(1-a^\ell b^\ell)]K^{\star}_{n-\ell}(a, b) \\
	&+ \sum_{\ell=1}^{n-1} [\varphi(1-a^\ell b^\ell)-\varphi(1-b^\ell a^{\ell-1})] F_{n-\ell}(a, b), 
	\end{aligned}\\[0.2cm]
F_n(a, b) &= [1-\varphi(1-a^n b^n)]
	\begin{aligned}[t]
	&+ \sum_{\ell=1}^{n-1} [\varphi(1-a^{\ell+1}b^\ell)-\varphi(1-a^\ell b^\ell)] F_{n-\ell}(a, b) \\
	&+ \sum_{\ell=1}^{n} [\varphi(1-a^\ell b^\ell)-\varphi(1-a^\ell b^{\ell-1})]K_{n+1-\ell}^{\star}(a, b), 
	\end{aligned}\\[0.2cm]
F_n^{\star}(a, b) &= [1-\varphi(1-b^n a^n)]
	\begin{aligned}[t]
	&+ \sum_{\ell=1}^{n-1} [\varphi(1-b^{\ell+1}a^\ell)-\varphi(1-a^\ell b^\ell)]F_{n-\ell}^{\star}(a, b)\\
	&+ \sum_{\ell=1}^{n} [\varphi(1-a^\ell b^\ell)-\varphi(1-b^\ell a^{\ell-1})]K_{n+1-\ell}(a, b).
	\end{aligned}
\end{align*}}%
Next, using that $\dsP[\m{x_0}{x_\ell}{}]=\dsP[\s{x_0}{x_\ell}]-\dsP[\s{x_0}{x_{\ell+1}}]$, we obtain, for every $n \geq 1$, 
{\allowdisplaybreaks
\begin{align*}
\dsP[\p{x_0}{x_{2n-1}}{o}]&=\dsP[\s{x_0}{x_{2n-1}}]+\sum_{\ell=1}^{2n-2} \dsP[\m{x_0}{x_\ell}{}]\dsP[\p{x_\ell}{x_{2n-1}}{o}]\\
&=\dsP[\s{x_0}{x_{2n-1}}]+\sum_{\ell=1}^{2n-2}\{\dsP[\s{x_0}{x_\ell}]-\dsP[\s{x_0}{x_{\ell+1}}]\}\dsP[\p{x_\ell}{x_{2n-1}}{o}], \\
\intertext{and}
\dsP[\p{x_0}{x_{2n}}{o}]&=\dsP[\s{x_0}{x_{2n}}]+\sum_{\ell=1}^{2n-1} \dsP[\m{x_0}{x_\ell}{}]\dsP[\p{x_\ell}{x_{2n}}{o}]\\
&=\dsP[\s{x_0}{x_{2n}}]+\sum_{\ell=1}^{2n-1}\{\dsP[\s{x_0}{x_\ell}]-\dsP[\s{x_0}{x_{\ell+1}}]\}\dsP[\p{x_\ell}{x_{2n}}{o}].
\end{align*}}%

We split the proof into two cases:
\begin{itemize}
\item[\textnormal{(i)}] If $(x_0, x_{2n-1})\in\VV_i\times\VV_j$ with $i \neq j$, then 
\begin{align*}
\nu_\eta(i, j, 2n-1)=\pi_\eta(i, j, 2n-1) &+ \sum_{\ell=1}^{n-1}[\pi_\eta(i, i, 2\ell)-\pi_\eta(i, j, 2\ell+1)]\nu_\eta(i, j, 2(n-\ell)-1)]\\
 &+ \sum_{\ell=1}^{n-1} [\pi_\eta(i, j, 2\ell-1)-\pi_\eta(i, i, 2\ell)]\nu_\eta(j, j, 2(n-\ell)).
\end{align*}

\item[\textnormal{(ii)}] If $(x_0, x_{2n})\in\VV_i\times\VV_i$, then
\begin{align*}
\nu_\eta(i, i, 2n)=\pi_\eta(i, i, 2n) &+ \sum_{\ell=1}^{n-1} [\pi_\eta(i, i, 2\ell)-\pi_\eta(i, j, 2\ell+1)]\nu_\eta(i, i, 2(n-\ell))\\
 &+ \sum_{\ell=1}^{n} [\pi_\eta(i, j, 2\ell-1)-\pi_\eta(i, i, 2\ell)]\nu_\eta(j, i, 2(n-\ell)+1).
\end{align*}
\end{itemize}
Using Equation~\eqref{F: prob open edge aleator}, the result follows by induction on $n$.
\end{proof}

\begin{proof}[Proof of Lemma~\ref{P: BPdom}]
From Lemma~\ref{L: Prob cam open}, the Galton--Watson branching process has mean number of progeny per individual given by 
\[ \dsE[Y_{1, n}] = (d_1d_2)^n \, \nu_{\eta}(1, 1, 2n) = (d_1d_2)^n \, F_n(\alpha(p), \beta(p)). \]
To obtain a lower bound for this expected value (not depending on the function $\varphi$), we truncate the initial configuration of the frog model.
We consider the modified initial configuration $\eta^{\prime}$ by 
\[\eta^{\prime}(x) = \ind_{\{\eta(x) \geq 1 \}}, \, x \in \VV.\]
Since the initial condition $\eta^{\prime}$ is dominated by $\eta$ in the usual stochastic order, it follows that
\[ \nu_{\eta}(1, 1, 2n) \geq \nu_{\eta^{\prime}}(1, 1, 2n). \]
But for the restricted frog model with initial configuration ruled by $\eta^{\prime}$, 
\[ \nu_{\eta^{\prime}}(1, 1, 2n) = F_n^{(q)}(\alpha(p), \beta(p)), \]
where $F_n^{(q)}(a, b)$ is given by the formulas stated in the proof of Lemma~\ref{L: Prob cam open} with the choice $\varphi(s) = 1 - q (1 - s)$.
Using these formulas and induction on $n$, we have that $F_n^{(q)}$ satisfies
\[ F^{(q)}_{n+1}(a, b)=a b [1+q(1-a)][1+q(1-b)]F^{(q)}_{n}(a, b), \, n \geq 1, \]
with initial condition $F_1^{(q)}(a, b)=q [ a b (1 + q (1-b))]$.
Consequently, for every $n \geq 1$, 
\[ F_n^{(q)}(a, b)=q[ab(1+q(1-b)]^n[1+q(1-a)]^{n-1}. \]
The result follows by noting that $\phi^{(\dd, q)}_n(p)$ given in~\eqref{E: prob offspring} is simply $F_n^{(q)}(\alpha(p), \beta(p))$.
\end{proof}

\begin{obs}
We underline that in the case when the random variable $\eta$ has Bernoulli distribution, the functions $K_n$, $F_n$, $K_n^{\star}$ and $F_n^{\star}$ defined in Lemma~\ref{L: Prob cam open} are polynomial in~$(a, b)$, but this does not hold in general.
Of course, by dealing only with this initial distribution, we could derive Equation~\eqref{E: prob offspring} from Definition~\ref{D: k Cam} in a more direct way.
Instead, we prefer to consider the general initial configuration in order to establish Lemma~\ref{L: Prob cam open}, for future reference.
\end{obs}

\subsection{Proofs of Theorem~\ref{T: UBA} and Corollary~\ref{C: UB bi-re}}
\label{SS: Proofs}

From Lemma~\ref{P: BPdom}, it follows that, by solving for each $n \geq 1$ the equation in $p$
\[ (d_1 d_2)^n \, \phi^{(\dd, q)}_n(p) = 1, \]
we obtain a sequence of upper bounds for the critical probability $\pc{\bbT_{\dd}, \eta}$.
So, for every $n \geq 1$, we define the function
\begin{equation}
\label{F: f_n}
f^{(\dd, q)}_n(p) = {\left[ \phi^{(\dd, q)}_n(p) \right]}^{1 / n}-\frac{1}{d_1d_2}.
\end{equation}
Notice that, for $p \in [0, 1]$, 
\[ \lim_{n \to \infty} f^{(\dd, q)}_n(p) = f^{(\dd, q)}(p), \]
where $f^{(\dd, q)}(p)$ is defined in~\eqref{F: Function-f}.
To prove the upper bound presented in Theorem~\ref{T: UBA}, we use the following result of Real Analysis, whose proof is included for the sake of completeness. 

\begin{lem}
\label{L: Conv}
Let $\{ f_n \}$ be a sequence of increasing, continuous real-valued functions defined on $[0, 1]$, such that $f_n (0) < 0$ and $f_n (1) > 0$ for every $n$.
Suppose that $\{ f_n \}$ converges pointwise as $n \to \infty$ to an increasing, continuous function $f$ defined on $[0, 1]$, and let $\tilde r_n$ be the unique root of $f_n$ in $[0, 1]$.
Then, there exists $\tilde r = \lim_{n \to \infty} \tilde r_n$ and $f (\tilde r) = 0$.
\end{lem}

\begin{proof}[Proof of Lemma~\ref{L: Conv}]
First, we observe that, under the stated conditions, the sequence $\{f_n\}$ converges uniformly to $f$ on $[0, 1]$ (see, e.g., \citet[Problem~9.19]{AB}).
Let $\rlinf = \liminf_{n \to \infty} \tilde r_n$ and $\rlsup = \limsup_{n \to \infty} \tilde r_n$.
Then, there exists a subsequence $\{\tilde r_{n_k}\}$ such that $\tilde r_{n_k} \to \rlinf$ as $k \to \infty$.
Consequently, $f_{n_k} (\tilde r_{n_k}) \to f(\rlinf)$ as $k \to \infty$ (by applying \citet[Problem~9.13a]{AB}). 
From this, it follows that $f(\rlinf) = 0$. 
Analogously, $f(\rlsup) = 0$.
But since $f$ is an increasing, continuous function with $f(0) \leq 0 \leq f(1)$, we conclude that $f$ has a unique root in the interval $[0, 1]$.
Hence, $\rlinf = \rlsup$, and the proof is complete.
\end{proof}

\begin{proof}[Proof of Theorem~\ref{T: UBA}]
It is straightforward to prove that there exists a positive integer $N=N(\dd)$ such that $\{f^{(\dd, q)}_n\}_{n\geq N}$ and $f^{(\dd, q)}$ defined in \eqref{F: f_n} and \eqref{F: Function-f} satisfy the conditions of Lemma~\ref{L: Conv}. 
Therefore, by defining $\tilde{p}_n(\dd, q)$ as the unique root of $f^{(\dd, q)}_n$ in the interval $[0, 1]$, it follows that 
\[ \lim_{n\to\infty}\tilde{p}_n(\dd, q)=\tilde{p}(\dd, q), \]
where the limit is the unique root of $f^{(\dd, q)}(p)$ in $(0, 1)$. 
Since $f^{(\dd, q)}_n(p)>0$ for every $p>\tilde{p}_{n}(\dd, q)$, from Lemma~\ref{P: BPdom}, we have that
\[\pc{\bbT_{\dd}, \eta}\leq \tilde{p}_{n}(\dd, q).\]
Taking $n \to \infty$, the results follows.
\end{proof}

\begin{proof}[Proof of Corollary~\ref{C: UB bi-re}]
If $\eta\equiv 1$, then the function $f^{(\dd, q)}(p)$ simplifies to
\[ f^{(\dd)}(p) = \alpha(p)\beta(p)[2-\alpha(p)][2-\beta(p)]-\frac{1}{d_1d_2}. \]
The upper bound $\tilde{p}(d_1, d_2)$ for $\pc{\bbT_{\dd}}$ is the unique root in $(0, 1)$ of this continuous and increasing function.
Let
\[ \bar{p} = \bar{p}(\dd) = \frac{1}{2}\sqrt{\frac{(d_1 + 1)(d_2 + 1)}{d_1 d_2}} \]
denote the upper bound given in the statement of Corollary~\ref{C: UB bi-re}.
The result follows from the fact that 
$f^{(\dd)}(\bar{p})\geq 0$.
This inequality can be verified with the help of a symbolic computation software.
For instance, in Mathematica, one can define the function $f^{(\dd)}(p)$ and $\bar{p}(\dd)$, and then use the command Reduce, to show that the statement
\[ f^{(\dd)}(\bar{p}(\dd)) \geq 0 \land (d_1 \geq 2 \land d_2 \geq 1) \lor (d_1 \geq 1 \land d_2 \geq 2) \]
is equivalent to
\[ (1 \leq d_1 < 2 \land d_2 \geq 2) \lor (d_1 \geq 2 \land d_2 \geq 1). \qedhere \]
\end{proof}

\section{Extinction of the process}
\label{S: Extinction}

To prove Theorem~\ref{T: SCE}, we define a coupled percolation process that dominates the frog model.
This percolation model is defined in a simple manner, not taking into account the trajectories of the frogs, so a coupling between the process on two graphs $\GG_1$ and $\GG_2$, with $\GG_1 \subset \GG_2$, is easily constructed.
The end of the proof relies on the idea employed to establish Theorem~1.2 in \citet{PT}.

\begin{proof}[Proof of Theorem~\ref{T: SCE}]
Let $\GG$ be an infinite tree of bounded degree, and let~$\Delta$ denote the maximum degree of $\GG$.
Recall that the extinction of the frog model on $\GG$ is equivalent to the finiteness of the cluster of the root in the percolation model on~$\GGO$ that is defined in Subsection~\ref{SS: FM Perc}.
Given a realization of $\FM{\GG}{p, \eta}$, we define, for every $x \in \VV$,
\begin{equation*}
\mcB_x:=
\left\{
\begin{array}{cl}
\bigcup_{k=1}^{\eta(x)} \{ y \in \VV: \dist(x,y) < \Xi_p^x(k) \} &\text{if } \eta(x) \geq 1, \\[0.2cm]
\{x\} &\text{otherwise}. 
\end{array}	\right.
\end{equation*}
Then, we consider the following oriented percolation model on~$\GGO$: the oriented edge $\vxy$ from vertex $x$ to vertex $y$ is declared to be open if $[y \in \mcB_x]$, and closed otherwise.
We call this model the \textit{disk-percolation model} on $\GG$, with parameter~$p$ and initial configuration ruled by $\eta$.
\citet{LR} study this model with $\eta \equiv 1$ on general graphs and spherically symmetric trees.
Since $\mcR_x \subset \mcB_x$ for every vertex $x$, we have that the frog model on $\GG$ dies out if the cluster of the root in the disk-percolation model on $\GG$ has finite size.

Now let $d = \max\{\Delta - 1, 2\}$.
Viewing $\GG$ as a subgraph of $\bbT_d$, we conclude that there is a natural coupling between the disk-percolation model on $\GG$ and on $\bbT_{d}$, in such a way that the cluster of the root on $\GG$ is finite, whenever it is finite on $\bbT_{d}$.
In addition, the disk-percolation model on $\bbT_{d}$ is dominated by a Galton--Watson branching process whose family size is distributed as $|\mcB_x \setminus \{x\}|$ (where $x$ is a fixed vertex of~$\bbT_d$).
Consequently, to finish the proof, it suffices to show that this branching process is subcritical for small enough $p > 0$.

The remainder of the proof follows the argument used to prove Theorem~1.2 in \citet{PT}, which we adapt here to our context.
For vertices $x$ and $y$ with $\dist(x,y) = k \geq 1$, by conditioning on~$\eta(x)$, we have
\begin{equation}
\label{F: DP}
\dsP[y \in \mcB_x] = \sum_{i=1}^{\infty} \rho_{i} \left[1 - (1 - p^k)^i\right].
\end{equation}
Let $\hat{k}(i,p)=\lfloor\log i/\log(1/p)\rfloor$, where $\lfloor x\rfloor$ denotes the largest integer
which is less than or equal to~$x$.
By using elementary calculus, one proves that there exists a constant $\hat{\beta}$ such that
\begin{equation}
\label{F: Phi}
1 - (1 - p^k)^i \leq \hat{\beta} \, p^{k-\hat{k}(i,p)-1} \, \text{ for every } k \geq \hat{k}(i,p) + 1.
\end{equation}
Given a vertex $x$ of $\bbT_d$, let $s_k(\bbT_d) = (d+1) \, d^{k-1}$ denote the cardinality of the set of vertices that are at distance $k$ from $x$.
Using~\eqref{F: DP} and \eqref{F: Phi}, we obtain that, for some positive constants $C_1$ and $C_2$,
\begin{align*}
\dsE |\mcB_x\setminus\{x\}| &= \sum_{k=1}^\infty s_k(\bbT_d) \sum_{i=1}^{\infty} \rho_{i} \left[1 - (1 - p^k)^i\right]\\
&\leq \sum_{i=1}^\infty \rho_i \, \Bigl[ \sum_{k=1}^{\hat{k}(i,p)}
s_k(\bbT_d) + \sum_{k=1}^\infty \hat{\beta} \, s_{\hat{k} + k}(\bbT_d) \, p^{k-1} \Bigr]\\
&\leq C_1 \displaystyle\sum_{i=1}^\infty \rho_i \, i^{\frac{\log d}{\log(1/p)}} + C_2.
\end{align*}
For some $p_0 > 0$, the last series converges uniformly in~$[0, p_0]$.
Hence, there exists a small enough~$p > 0$ such that $\dsE|\mcB_x\setminus\{x\}|<1$, completing the proof.
\end{proof}

Our strategy to show Theorem~\ref{T: LB} is to compare the frog model on $\bbT_{\dd}$ with a suitable subcritical multitype Galton--Watson branching process.
The method is inspired by the idea of comparing the frog model with an ordinary branching process that is used by \citet{PT} to prove Proposition~\ref{P: LB Alves}.

\begin{proof}[Proof of Theorem~\ref{T: LB}]
Consider the multitype Galton--Watson branching process with two types, defined as follows.
It starts with zero particles of type $1$, and a random number of particles of type $2$, which is given by the sum of $(d_1 + 2)$ independent copies of $\eta$.
For $i = 1, 2$ and $k_1, k_2 \in \bbN_0$, let $\mathbf{p}^{(i)}(k_1, k_2)$ denote the probability that a type~$i$ particle produces $k_1$ particles of type~$1$, and $k_2$ particles of type~$2$.
Recall that $\rho_k := \dsP[\eta = k]$, and suppose that the progeny distribution is given by
\begin{alignat*}{3}
\mathbf{p}^{(1)}(0, 0) &= 1-p, \qquad &\mathbf{p}^{(1)}(0, 1) &= \frac{p (1 + d_1 \rho_0)}{d_1+1}, 
\qquad &\mathbf{p}^{(1)}(0, k) &= \frac{pd_1 \rho_{k-1}}{d_1+1}, \; k=2, 3, \dots \\[0.2cm]
\mathbf{p}^{(2)}(0, 0) &= 1-p, &\mathbf{p}^{(2)}(1, 0) &= \frac{p (1 + d_2 \rho_0)}{d_2+1}, 
&\mathbf{p}^{(2)}(k, 0) &= \frac{pd_2 \rho_{k-1}}{d_2+1}, \; k=2, 3, \dots
\end{alignat*}

Notice that, in the frog model on~$\bbT_{\dd}$, at time $n = 1$, the number of awake frogs placed on type $1$ and type $2$ vertices is stochastically smaller than the initial configuration of the multitype Galton--Watson process just defined.
Furthermore, in the frog model, every vertex with at least one awake frog at time $n \geq 1$ has at least one neighbor vertex whose original frogs have been awoken prior to time $n$.
Thus, the multitype Galton--Watson process dominates the frog model on $\bbT_{\dd}$, in the sense that the frog model becomes extinct if this process does. 

Now let $M:=(m_{i, j})_{i, j\in\{1, 2\}}$ be the first moment matrix, i.e., $m_{i, j}$ is the expected number of type~$j$ offspring of a single type~$i$ particle in one generation. 
Since the number of types is finite, it is well known (see~\citet[Chap.~V]{BP}) that the multitype Galton--Watson process dies out almost surely if and only if the largest eigenvalue $\lambda(M)$ of the matrix $M$ is less than~$1$.
In our case, the first moment matrix is given by
\begin{equation*}
M = \begin{bmatrix}
0 & \frac{p}{d_1+1}(1+d_{1}(\dsE\eta+1)) \\[0.2cm]
\frac{p}{d_2+1}(1+d_{2}(\dsE\eta+1)) & 0
\end{bmatrix}.
\end{equation*}
An elementary calculation shows that if 
\[ p<\sqrt{\frac{(d_1+1) (d_2+1)}{[d_1(\dsE\eta+1)+1] [d_2(\dsE\eta+1)+1]}}, \] 
then $\lambda(M) < 1$, therefore the multitype Galton--Watson process dies out almost surely. 
Consequently, the same happens to the frog model. 
\end{proof}

\section*{Acknowledgements}

The authors are grateful to two anonymous referees, who carefully read the paper and made many valuable questions and suggestions that helped us to improve it.

\bibliography{bibPTBT}

\begin{thebibliography}{19}
\providecommand{\natexlab}[1]{#1}
\providecommand{\url}[1]{\texttt{#1}}
\expandafter\ifx\csname urlstyle\endcsname\relax
  \providecommand{\doi}[1]{doi: #1}\else
  \providecommand{\doi}{doi: \begingroup \urlstyle{rm}\Url}\fi

\bibitem[Aliprantis and Burkinshaw(1999)]{AB}
C.~D. Aliprantis and O.~Burkinshaw.
\newblock \emph{Problems in {R}eal {A}nalysis. {A} {W}orkbook with
  {S}olutions}.
\newblock Academic Press, San Diego, CA, 2nd edition, 1999.

\bibitem[Alves et~al.(2001)Alves, Machado, Popov, and Ravishankar]{STR}
O.~S.~M. Alves, F.~P. Machado, S.~Popov, and K.~Ravishankar.
\newblock The shape theorem for the frog model with random initial
  configuration.
\newblock \emph{Markov Process. Related Fields}, 7\penalty0 (4):\penalty0
  525--539, 2001.

\bibitem[Alves et~al.(2002{\natexlab{a}})Alves, Machado, and Popov]{PT}
O.~S.~M. Alves, F.~P. Machado, and S.~Popov.
\newblock Phase transition for the frog model.
\newblock \emph{Electron. J. Probab.}, 7\penalty0 (16):\penalty0 21 p.,
  2002{\natexlab{a}}.

\bibitem[Alves et~al.(2002{\natexlab{b}})Alves, Machado, and Popov]{ST}
O.~S.~M. Alves, F.~P. Machado, and S.~Popov.
\newblock The shape theorem for the frog model.
\newblock \emph{Ann. Appl. Probab.}, 12\penalty0 (2):\penalty0 534--547,
  2002{\natexlab{b}}.

\bibitem[Athreya and Ney(1972)]{BP}
K.~B. Athreya and P.~E. Ney.
\newblock \emph{Branching Processes}.
\newblock Springer-Verlag, New York, 1972.

\bibitem[Fontes et~al.(2004)Fontes, Machado, and Sarkar]{Mono}
L.~R. Fontes, F.~P. Machado, and A.~Sarkar.
\newblock The critical probability for the frog model is not a monotonic
  function of the graph.
\newblock \emph{J. Appl. Probab.}, 41\penalty0 (1):\penalty0 292--298, 2004.

\bibitem[Gallo and Rodriguez(2018)]{FMRT}
S.~Gallo and P.~M. Rodriguez.
\newblock Frog models on trees through renewal theory.
\newblock \emph{J. Appl. Probab.}, 55\penalty0 (3):\penalty0 887--899, 2018.

\bibitem[Grimmett(1999)]{Grim}
G.~Grimmett.
\newblock \emph{Percolation}.
\newblock Springer-Verlag, Berlin, 2nd edition, 1999.

\bibitem[Hoffman et~al.(2016)Hoffman, Johnson, and Junge]{JJHa}
C.~Hoffman, T.~Johnson, and M.~Junge.
\newblock From transience to recurrence with {P}oisson tree frogs.
\newblock \emph{Ann. Appl. Probab.}, 26\penalty0 (3):\penalty0 1620--1635,
  2016.

\bibitem[Hoffman et~al.(2017)Hoffman, Johnson, and Junge]{JJHb}
C.~Hoffman, T.~Johnson, and M.~Junge.
\newblock Recurrence and transience for the frog model on trees.
\newblock \emph{Ann. Probab.}, 45\penalty0 (5):\penalty0 2826--2854, 2017.

\bibitem[Lalley and Sellke(1998)]{LSCP}
S.~P. Lalley and T.~Sellke.
\newblock Limit set of a weakly supercritical contact process on a homogeneous
  tree.
\newblock \emph{Ann. Probab.}, 26\penalty0 (2):\penalty0 644--657, 1998.

\bibitem[Lebensztayn and Rodriguez(2008)]{LR}
E.~Lebensztayn and P.~M. Rodriguez.
\newblock The disk-percolation model on graphs.
\newblock \emph{Statist. Probab. Lett.}, 78\penalty0 (14):\penalty0 2130--2136,
  2008.

\bibitem[Lebensztayn and Utria(2019)]{NUB}
E.~Lebensztayn and J.~Utria.
\newblock A new upper bound for the critical probability of the frog model on
  homogeneous trees.
\newblock To appear in \emph{J. Stat. Phys.}, 2019.

\bibitem[Lebensztayn et~al.(2005)Lebensztayn, Machado, and Popov]{IUB}
E.~Lebensztayn, F.~P. Machado, and S.~Popov.
\newblock An improved upper bound for the critical probability of the frog
  model on homogeneous trees.
\newblock \emph{J. Stat. Phys.}, 119\penalty0 (1-2):\penalty0 331--345, 2005.

\bibitem[Popov(2001)]{FIRE}
S.~Popov.
\newblock Frogs in random environment.
\newblock \emph{J. Statist. Phys.}, 102\penalty0 (1-2):\penalty0 191--201,
  2001.

\bibitem[Popov(2003)]{FIS}
S.~Popov.
\newblock Frogs and some other interacting random walks models.
\newblock In \emph{Discrete random walks ({P}aris, 2003)}, Discrete Math.
  Theor. Comput. Sci. Proc., AC, pages 277--288. Assoc. Discrete Math. Theor.
  Comput. Sci., Nancy, 2003.

\bibitem[Ram\'{i}rez and Sidoravicius(2004)]{STCT}
A.~F. Ram\'{i}rez and V.~Sidoravicius.
\newblock Asymptotic behavior of a stochastic combustion growth process.
\newblock \emph{J. Eur. Math. Soc. (JEMS)}, 6\penalty0 (3):\penalty0 293--334,
  2004.

\bibitem[Rosenberg(2018)]{JR}
J.~Rosenberg.
\newblock Recurrence of the frog model on the \(3,2\)-alternating tree.
\newblock \emph{ALEA Lat. Am. J. Probab. Math. Stat.}, 15\penalty0
  (2):\penalty0 811--836, 2018.

\bibitem[Telcs and Wormald(1999)]{TW}
A.~Telcs and N.~C. Wormald.
\newblock Branching and tree indexed random walks on fractals.
\newblock \emph{J. Appl. Probab.}, 36\penalty0 (4):\penalty0 999--1011, 1999.

\end{thebibliography}
\bibliographystyle{plainnat}

\end{document}